\documentclass[12pt]{amsart}
\usepackage{amssymb}
\usepackage{verbatim}
\usepackage{amscd}

\theoremstyle{plain}
\newtheorem{theorem}{Theorem}
\newtheorem{corollary}{Corollary}
\newtheorem{proposition}{Proposition}
\newtheorem{lemma}{Lemma}

\theoremstyle{definition}
\newtheorem{definition}{Definition}

\theoremstyle{remark}
\newtheorem{remark}{Remark}
\newtheorem{example}{Example}
\newtheorem*{notation}{Notation}

\newcommand{\ti}{\tilde}
\newcommand{\wti}{\widetilde}

\newcommand{\ol}[1]{\overline{#1}}

\newcommand{\vil}{\varinjlim}

\def\bbC{\mathbb C}

\def\bbN{\mathbb N}

\def\bbR{\mathbb R}

\def\bbZ{\mathbb Z}

     \newcommand{\sA}{\mathcal A}
     \newcommand{\sB}{\mathcal B}

     \newcommand{\sG}{\mathcal G}
     \newcommand{\sH}{\mathcal H}
     
     \newcommand{\sK}{\mathcal K}

     \newcommand{\sS}{\mathcal S}

\newcommand{\al}{\alpha}
\newcommand{\be}{\beta}
\newcommand{\ep}{\epsilon}
\newcommand{\vpi}{\varphi}
\newcommand{\si}{\sigma}

\newcommand{\de}{\delta}
\newcommand{\om}{\omega}
\newcommand{\io}{\iota}
\newcommand{\Ga}{\Gamma}
\newcommand{\ga}{\gamma}

\begin{document}
% topmatter
\title[Semigroups and operator algebras]{Operator algebras and representations from commuting semigroup actions}

\author[Duncan]{Benton L.\ Duncan}
\address{Department of Mathematics\\
North Dakota State University\\
Fargo, North Dakota\\
USA} \email{benton.duncan@ndsu.edu}

\author[Peters]{Justin R.\ Peters$^{\ddag}$}
\address{Department of Mathematics \\
Iowa State University \\
Ames, Iowa\\
USA} \email{peters@iastate.edu}

\thanks{{$\ddag$} The author acknowledges partial support from the National Science Foundation, DMS-0750986}


\begin{abstract}
Let $\sS$ be a countable, abelian semigroup of continuous
surjections on a compact metric space $X$. Corresponding to this
dynamical system  we associate two operator algebras, the tensor
algebra, and the semicrossed product. There is a unique smallest
C$^*$-algebra into which an operator algebra is completely
isometrically embedded, which is the C$^*$-envelope.  
The C$^*$-envelope of the tensor algebra is a crossed product C$^*$-algebra. We also study
two natural classes of representations, the left regular
representations and the orbit representations.  The first is Shilov,
and the second has a Shilov resolution.
\end{abstract}

\maketitle

\section{INTRODUCTION}

Let $X$ be a compact metric space, and $\sS$ an abelian semigroup
and let $\si$ be a map of $\sS$ into the set of continuous,
surjective maps of $X \to X$, which we assume to be a semigroup
isomorphism. From this dynamical system $(X, \si, \sS)$ we construct
two operator algebras: the tensor algebra, and the semicrossed
product.

If the semigroup $\sS$ is a group, then the tensor algebra and the
semicrossed product coincide with the crossed product,
$C(X)\rtimes_{\si} \sS.$ Our interest is in dealing with
noninvertible dynamics, so we will assume that the semigroup $\sS$
is not a group.

Work on such problems began with single-variable dynamics
\cite{Arveson-Josephson}, \cite{Peters:semicrossed} and many others.
Work with multivariate dynamics is more recent. This paper is in a
sense a counterpoint to the important contribution of Davidson and
Katsoulis \cite{DK:multivariate}, in which they studied various
operator algebras that could be considered multivariate analogues of
the (single variable) semicrossed product, and developed the
dilation theory and isomorphism properties of these algebras. \cite{KK} and \cite{F} are
also closely related. In
\cite{DKM}, Donsig, Katavolos and Manoussos give a precise description
of the Jacobson radical of semicrossed products, where the semigroup
is $\bbZ^d_+.$

While the point of view of C$^*$-dynamical systems mostly deals with
group actions on C$^*$-algebras, Exel \cite{Exel:endo} and Exel and
Renault \cite{Exel-Renault:groupoid} consider noninvertible
dynamical systems, such as local homeomorphisms on a compact space.
There is the additional feature of the transfer operator, which is
not present here. Nevertheless it is interesting to compare their
approach to the C$^*$-algebra which arises naturally in our context
as the C$^*$-envelope of the tensor algebra.

We begin by constructing an algebra $\sA_0$ which contains operators
$S_s$ for $s $ an element of the semigroup $\sS$, and functions $f
\in C(X),$ the continuous complex valued functions on $X$, subject
to the covariance condition
\[ f \,S_s = S_s\, f\circ \si_s \, . \]
An element of the algebra $\sA_0$ has the form $ \sum_s S_s\, f_x, $
where the sum is finite. We study classes of representations of this
algebra. One natural class of representations  arises from the
left regular representation on the Hilbert space $\ell_2(\sS)$ and the
evaluation map of functions at a point $x \in X.$ These
representations, denoted by $\pi$, represent the operators $S_s$ as
isometries, and they separate the points of $\sA_0.$  Completing
$\sA_0$ in the norm determined by these representations yields an
algebra $\sA(X, \sS)$ which we call the left regular algebra.

Another class of representations we study we call orbit
representations.  These are similar to the representations $\pi$,
except they act on the orbit of a point $x \in X.$   We denote the orbit
representations by $\rho$.  While orbit representations have been studied in the context of group actions,
the semigroup setting presents features not present when dealing with goup actions.  We
show these representations are associated with cocycles, and indeed
there is a one-to-one correspondence between the orbit
representations and the orbit cocycles.

We have defined two nonselfadjoint operator algebras arising from
the dynamical system $(X, \si, \sS).$ One is the tensor algebra,
already mentioned.  The other is the semicrossed product.  This is
the completion of the $\sA_0$ in the norm arising from considering
all isometric covariant representations (Definition~\ref{d:rep}).
However we have no tools to characterize all such representations,
so there is little we can say about such algebras.

Davidson and Katsoulis \cite{DK:multivariate} use the general approach of Katsura  \cite{Katsura} and Muhly and Solel
\cite{Muhly-Solel2} to obtain the tensor algebra and its C$^*$-envelope via C$^*$-correspondences.
Our approach to the C$^*$-envelope,
done in \cite{Peters:envelope} for the single variable
setting, yields a more tangible result, yet is only available in a
restricted context.

While the enveloping group $\sG$ containing the semigroup $\sS,$ is
easily obtained as $\sG = \sS - \sS,$ there need not be any
connection between the abstract group $\sG$ and mappings on the
compact metric space $X$. In section~\ref{s:extension} we construct
a compact metric space $\wti{X}$ on which the group $\sG$ acts by
homeomorphisms $\wti{\si_g} \ (g \in \sG)$ and a continuous
surjection $p: \wti{X} \to X$ which ``intertwines'' this group
action with the original semigroup action. Theorem~\ref{t:cstarenv}
shows that the C$^*$-envelope of the tensor algebra $\sA(X, \sS)$ is
identified with the C$^*$-crossed product
$C(\wti{X})\rtimes_{\wti{\si}} \sG.$

The description of the C$^*$-envelope in Theorem~\ref{t:cstarenv}
also yields some information about the left regular representations
$\pi$ and the (left regular) orbit representations $\rho$. We are
able to show that the representations $\pi$ are Shilov, and that the
left regular orbit representations have a Shilov resolution.

We should comment on the relation of our results with those of
\cite{DK:multivariate}. They consider actions of the free semigroup
on $n$-generators (for fixed $n \in \bbN$), whereas the semigroups we consider
are abelian. Even though we do not deal with specific examples of
dynamical systems in this paper, it is also worth noting that there
are actions which fall within our context which are not homomorphic
images of free finitely generated semigroups: in
\cite{Peters:semigroup}, Example 5, there is an action of the
semigroup of non-negative dyadic rationals on a compact metric space
$X$ by local homeomorphisms. We should also note that there is
relatively little overlap of our results with \cite{DK:multivariate}. Because Davidson and Katsoulis deal with
finitely many coordinates, they are able to obtain a number of
dilation results. However the example of Parrott \cite{Parrott} of
three commuting contractions which do not admit a unitary dilation
illustrates the inherent difficulty of a general dilation theory in
our setting.
What we are able to achieve, is a dilation of the
commuting contractions $\rho(S_s)\ (s \in \sS)$ to unitaries.  While there are a number of positive results in the literature, such as the dilation results for $n-$tuples of doubly commuting contractions, we are not aware that our dilation theory
overlaps with other such results.

Since this paper was written, we were given a preprint of Davidson, Fuller and Kakariadis (\cite
{DFK}) in which they obtain
another proof of our theorem~\ref{t:cstarenv}.

\section{BACKGROUND AND NOTATION}
\textbf{Standing Hypothesis} Throughout the paper, $ \sS$ will
denote an abelian semigroup with cancellation, and identity element,
denoted by $0$. The semigroup operation will be written as addition.
The intersection of all abelian groups which contain $\sS$ will be
written as $\sG = \sS - \sS.$

Some of the constructions, such as the left regular representation $\pi_x,$ and the orbit reprsentation
$\rho_x,$ do not require the commutativity of the semigroup. The tools we employ, such as integrating
over the compact dual group $\Ga$ of the group $\sG = \sS - \sS,$ do use commutativity. For that reason, we assume
commutativity of $\sS$ throughout the paper.

 We assume that $\sS$ acts on a compact metric space.
Thus, there is a homomorphism, denoted by $\si$, from $\sS$ into the
semigroup of continuous, surjective maps of $X \to X.$ There is no
loss of generality by assuming that $\si$ is a semigroup isomorphism
(onto its image), which we will do. Furthermore, we assume that
$\sS$ is not a group, for otherwise nothing new is achieved.
However, it may be the case that $\sS$ contains a nontrivial group
$\sS \cap -\sS.$ The triple $(X, \si, \sS)$ will be called a
dynamical system.

We will not keep repeating these assumptions in the statements of
our results.

\section{SEMICROSSED PRODUCTS} \label{s:semicrossed}

Let $\sA_0$ be the algebra generated by $C(X)$ together with symbols
$ S_s,\ s \in \sS$ and subject to the relations \\
\parbox{2cm}{ \begin{align*} f S_s = S_s f\circ\si_s, &\quad s \in
\sS,\ f \in C(X) \\ S_{s+t} = S_{s}S_{t}, &\quad s, t \in \sS
\end{align*}}\hfill \parbox{1cm}{(\dag)}

Thus a typical element of the algebra has the form
\[ \sum_{s \in \sS} S_s f_s \]
where the sum is finite.

Let $\Ga$ be the dual group of $\sG$.

\begin{definition} \label{d:gauge} Define the \emph{Gauge automorphism}
$\tau_{\ga}\ (\ga \in \Ga$) on $\sA_0$ by \[ \tau_{\ga} (\sum_s
S_s\,f_s) = \sum_s \langle \ga, s \rangle \,S_s\,f_s .\]
\end{definition}

Define the projections $P_s,\ s \in \sS$: for $F \in \sA_0,$
\[ P_s(F) = \int_{\Ga} \tau_{\ga}(F)\langle -\ga, s\rangle\,d\ga \]
where $d\ga$ is Haar measure on the compact group $\Ga.$ Note we are
considering the semigroup $\sS$ and the group $\sG$ as discrete
groups, and so $\Ga$ is a compact abelian group.

Note that if $F = \sum_s S_s\,f_s \in \sA_0,$ then $P_{s_0}(F)$ is
equal to either $S_{s_0}\, f_{s_0}$ or $0$ if $s_0$ is not in the
sum.

\begin{definition} \label{d:rep} We say
that a representation
\[ \pi: \sA_0 \to \sB(\sH) \] with the following properties:
\begin{enumerate}
\item $\pi(S_s)$ is an isometry (resp., a contraction) in $\sB(\sH)$ for all $s \in \sS;$
\item $\pi(S_0) = I;$
\item $\pi|_{C(X)}$ is a C$^*$-representation.
\end{enumerate}
is an isometric (resp., a contractive) covariant representation of
the pair $(C(X), \sS)$.
\end{definition}
Observe that $C(X)$ is embedded in $\sA_0$ by the map $ f \mapsto S_0 f .$

In \cite{DK:multivariate} Davidson and Katsoulis consider four sets
of conditions on representations. But two of those conditions do not
have a direct translation into this general context--namely, row
contractive and row isometric, since our semigroup need not be
freely generated by finitely many $S_s$.

\begin{definition} Let $C(X) \rtimes_{\si} \sS$ denote the semicrossed product
algebra; that is, the completion of $\sA_0$ with respect to the norm
\[ ||F|| = \sup_{\pi} ||\pi(F)||\]
for $F \in \sA_0$, where the supremum is over all representations
$\pi$ satisfying properties $(1), (2), (3)$ of the definition.
\end{definition}

\section{THE LEFT REGULAR ALGEBRA} \label{s:leftreg}

We now define a class of representations of $\sA_0$ which will play an
important role in what follows.

Given $x \in X$ and $\ga \in \Ga$ define a representation $\pi_{x,
\ga} $ of $\sA_0$ on the Hilbert space $\ell_2(\sS)$ as follows: Let $\xi_s \in
\ell_2(\sS)$ be given by
\[ \xi_s(t) =
\begin{cases}
1 \text{ if } t = s;\\
0 \text{ otherwise.}
\end{cases}
\]
It suffices to define $\pi_{x, \ga}(f),\ f \in C(X),$ and  $\pi_{x,
\ga}(S_t)$ on the vectors $\xi_s$ since linear combinations of such
vectors are dense. Set
\[ \pi_{x, \ga}(f)\xi_s = f\circ\si_s(x)\xi_s \]
and
\[ \pi_{x, \ga}(S_t)\xi_s = \langle\ga, t\rangle \xi_{t+s} .\]

It is a routine calculation to  verify that $\pi_{x, \ga}$ respects
the relations $\dag$.

The adjoint is given by

\[
\pi_{x, \ga}(S_t)^*\xi_s =
\begin{cases}
\ol{\langle\ga, t\rangle}\xi_{u} \text{ if } s = t + u \text{ for some } u \in \sS \\
0 \text{ otherwise. }
\end{cases}
\]
so that $ \pi_{x, \ga}(S_t)^* \pi_{x, \ga}(S_t)\xi_s = \xi_s$ for
all $s \in \sS,$ and since the set $\{ \xi_s :\ s \in \sS\}$ is an
orthonormal basis for $\ell_2(\sS),$ it follows $ \pi_{x,
\ga}(S_t)^* \pi_{x, \ga}(S_t) = I .$

It is obvious that $\pi_{x, \ga}$ is a $*$-representation when
restricted to $C(X).$ Thus it is an isometric covariant
representation.

\begin{remark} \label{r:regrep} Notice that the unitary given by $ \xi_s \mapsto \langle  - \ga, s \rangle \xi_s$ provides a unitary equivalence between the representation $ \pi_{x,\ga}$ and $ \pi_{x, 0 }$ and the representations $\pi_{x, 0}$ are the semigroup analogue of
the regular representations of crossed products, coming from the one
dimensional evaluation representations, as in \cite[7.7]{Pet}.
\end{remark}

\begin{lemma} \label{l:regseparating} Let $F \in \sA_0,\ F \neq 0.$
Then for some $x \in X,\ \ga \in \Ga,\  \pi_{x, \ga}(F) \neq 0.$
\end{lemma}

\begin{proof} By the remark it suffices to consider $ \ga $ the trivial character.  We write $F = \sum_{s \in I} S_s\, f_s$ where $I$ is a
finite subset of $\sS,$ and such that $f_s \neq 0$ for $s \in I.$
Let $u \in \sS,\ s_0 \in I$ and compute
\begin{align*}
\int_{\Ga} \pi_{x, \ga}(F)\xi_u \,  d\ga &=
\sum_{s\in
I} \int_{\Ga} \pi_{x, \ga}(S_s f_s)\xi_u\, d\ga \\
    &= \sum_{s\in I} \int_{\Ga} \langle\ga, s\rangle
    f_s(\si_u(x))\xi_{s+u} d\ga \\
    &= f_{s_0}(\si_u(x))\xi_{s+u}
\end{align*}
We may choose $x \in X$ and $u \in \sS$ such that $f_{s_0}(\si_u(x))
\neq 0.$ Thus, there is a choice of $ x \in X$ and $\ga \in \Ga$ for
which $\pi_{x, \ga}(F) \neq 0.$
\end{proof}

\begin{corollary} The class of representations $\pi_{x, \ga}, \ (x,
\ga) \in X\times \Ga,$ separates the elements of $\sA_0.$
\end{corollary}

\begin{definition} \label{d:leftreg}
We define the left regular algebra $\sA(\sS, X)$ to be the
completion of $\sA_0$ in the norm of the representation

\[ {\oplus_{(x, \ga) \in X\times \Ga}\, \pi_{x, \ga}}. \]
\end{definition}

\begin{remark} \label{r:leftreg}
The representations $\pi_{x, \ga},$ initially defined on the algebra $\sA_0,$ admit a unique extension to the left regular
algebra $\sA(\sS, X).$ The extended representations will also be denoted $\pi_{x, \ga}.$
\end{remark}

\begin{remark} \label{r:unitaryequiv}
In light of remark~\ref{r:regrep},
the norm on $\sA(\sS, X)$ could be defined using the subclass of representations $\pi_{x, 0}.$
\end{remark}

\begin{notation} We will write $\pi_x$ for $\pi_{x, 0}.$  In other words, if $\ga$ is the trivial character $0,$
we will omit the $0.$
\end{notation}

Let $F \in \sA(\sS, X),\ ||F|| = 1,$ and suppose that for all $u \in
\sS,\ P_u(F) = 0.$ Now there are $x \in X,$ and unit
vectors $\xi,\ \eta \in \ell_2(\sS)$ for which
\[ |(\pi_{x}(F)\xi, \eta)| > \frac{1}{2}||F|| .\] H
Here we are making use of Remark~\ref{r:unitaryequiv}, that it is sufficient to consider the
representations $\pi_x.$
Hence there exist $s,\ t \in \sS$ such that
\[ \ep := |(\pi_{x}(F)\xi_s, \xi_t)| > 0 .\]

Let $G \in \sA_0$ be such that $||F - G|| < \de$, where $ 0 < \de <
\ep/2 .$  Then
\begin{align*}
|(\pi_{x}(G)\xi_s, \xi_t)| &\geq |(\pi_{x}(F)\xi_s,
\xi_t)| - ||F - G|| \\
    &\geq \ep - \de \\
    &> \ep/2.
\end{align*}
Express $G = \sum_u S_u\,f_u.$ Now $|(\pi_{x}(G)\xi_s, \xi_t)|
> 0$ implies for $u = t - s \in \sS,\ f_u \neq 0.$ Thus for $u
= t - s$ we have
\[ P_u(G - F) = P_u(G) = S_u\,f_u .\]
Now \[|(\pi_{x}(G)\xi_s, \xi_t)| = |f_u(\si_s(x))| >\ep/2,\] so
that $||P_u(G)|| > \ep/2.$  On the other hand,
\[ ||P_u(G)|| = ||P_u(F - G)|| \leq ||F - G|| < \ep/2.\]
We have shown the following:
\begin{proposition} \label{p:nonzeroproj} If
$F$ is a nonzero element of $\sA(\sS, X)$  then there exists $u \in
\sS$ such that $P_u(F) \neq 0.$
\end{proposition}

%%%%%%%%%%%%%%%%%%%%%%%%%%%%%%%%%%%%%%%%%%%%%%%%%%%%%%%%%%%%%%%%%%%%%%%%%%%%%%%%%%%%%%%%%%%%%%%%%%%%%%%%%%%%%%%
%%%%%%%%%%%%%%%%%%%%%%%%%%%%%%%%%%%%%%%%%%%%%%%%%%%%%%%%%%%%%%%%%%%%%%%%%%%%%%%%%%%%%%%%%%%%%%%%%%%%%%%%%%%%%%%%%
%%%%%%%%%%%%%%%%%%%%%%%%%%%%%%%%%%%%%%%%%%%%%%%%%%%%%%%%%%%%%%%%%%%%%%%%%%%%%%%%%%%%%%%%%%%%%%%%%%%%%%%%%%%%%%%%%

\subsection{ORBIT REPRESENTATIONS}
Next we define another class of representations of $\sA_0,$ which we
call orbit representations. Fix $x \in X$ and let $\sS(x)$ denote
the orbit of $x,$ namely, $\sS(x) = \{ \si_s(x):\ s \in \sS\}.$

\begin{definition} \label{d:orbitcocycle} A function $\mu : \sS
\times \sS(x) \to \bbC$ is an \emph{orbit cocycle} if it satisfies
\begin{enumerate}
\item For each $t \in \sS$ any $y \in \sS(x)$
\[ \sum_{\si_t(y_j) = \si_t(y)} |\mu(t, y_j)|^2 \leq 1 \]
\item (cocycle condition) For each $s,\ t \in \sS,$ and $y \in
\sS(x)$
\[ \mu(s+t, y) = \mu(t, y)\, \mu(s, \si_t(y))  .\]
\end{enumerate}
We may also write $\mu_x$ to emphasize the dependence on the point
$x \in X.$
\end{definition}

We will define the orbit representations $\rho_{x, \mu} $ of the algebra $\sA_0$ on $\ell_2(\sS(x)).$

Let
$\xi_y$ be the function
\[ \xi_y(w) =
\begin{cases}
1, \text{ if } w = y \\
0, \text{ otherwise. }
\end{cases}\]

Define, for $f \in C(X),\ y \in \sS(x)$
\[ \rho_{x, \mu}(f)\xi_y = f(y)\xi_y \]
and
\[ \rho_{x, \mu}(S_t)\xi_y = \mu(t, y) \xi_{\si_t(y)} .\]
Let $\xi = \sum a_i \xi_{y_j}$ be a unit vector in $\ell_2(\sS(x))$
such that $\si_t(y_j) = \si_t(y)$ for all $j.$
\[ \rho_{x, \mu}(S_t)\xi = (\sum a_j \mu(t, y_j))\xi_{\si_t(y)} .\]
Hence
\begin{align*}
 ||\rho_{x, \mu}(S_t)\xi||^2 &= \left|\sum a_j \mu(t, y_j)\right|^2 \\
    &\leq \left(\sum |a_j|^2\right) \left(\sum |\mu(t, y_j)|^2\right).
\end{align*}
Since $\xi$ is a unit vector, $\sum |a_j|^2 = 1.$ Hence, if
$\rho_{x, \mu}(S_t)$ is to be contractive, we must have that $\sum
|\mu(t, y_j)|^2 \leq 1.$ On the other hand, let us note that this
condition is sufficient for $\rho_{x, \mu}(S_t)$ to be contractive.
Consider the dense set of vectors which are linear combinations of
vectors $\xi$ of the above form. Say $\eta = \sum b_k \xi_k,$ where
for each $k, \ \rho_{x, \mu}(S_t)\xi_k $ is a multiple of
$\xi_{u_k}$ for some $u_k \in \sS,$ where the $u_k$ are distinct
elements of $\sS,$ the $\xi_k$ are unit vectors, and $\sum |b_k|^2 =
1.$  Then by the above we have that
\begin{align*}
||\rho_{x, \mu}(S_t)\eta||^2 &= \left|\left|\sum \rho_{x,
\mu}(S_t)b_k \xi_{k}\right|\right|^2
\\
    &\leq \left|\left|\sum b_j \xi_{u_k}\right|\right|^2 \\
    &\leq 1
\end{align*}

Additionally we have, for $s,\ t \in \sS$ and $ y \in \sS(x)$
\begin{align*}
\rho_{x, \mu}(S_{s+t})\xi_y &= \rho_{x, \mu}(S_s S_t)\xi_y \\
    &= \rho_{x, \mu}(S_s) \rho_{x, \mu} (S_t)\xi_y \\
    &= \rho_{x, \mu}(S_s) \mu(t, y) \xi_{\si_t(y)} \\
    &= \mu(t, y) \mu(s, \si_t(y)) \xi_{\si_{s+t}(y)} \\
    &= \mu(s+t, y) \xi_{\si_{s+t}(y)}
\end{align*}

To conclude that $\rho_{x, \mu}$ is a representation, we need
$\rho_{x, \mu} (f S_s) = \rho_{x, \mu}(S_s f\circ \si_s),\ s \in
\sS,\ f \in C(X).$ But that is a routine calculation.

We summarize this as
\begin{corollary} \label{c:orbitcocycle}
Orbit representations are contractive covariant representations.
Furthermore, there is a one-to-one correspondence between orbit
representations and orbit cocycles.
\end{corollary}

\begin{remark} \label{r:orbitcocycle}

Let $\mu$ be an orbit cocycle, and $\ga \in \Ga$. Then $\ga \mu$ is
also an orbit cocycle.  That is, $\ga \mu(t, y) = \langle\ga,
t\rangle\mu(t, y).$
\end{remark}

To address the question of what can be said about the existence of
orbit cocycles we need a definition from \cite{Peters:semigroup}.
\begin{definition} \label{d:cocycle} A \emph{cocycle} for a
dynamical system $(X, \si, \sS) $ is a map $\om: \sS \times X \to
\bbR$ such that
\begin{enumerate}
\item $\om(s, x) \geq 0 $ for all $s \in \sS,\ x \in X;$
\item For each $y \in X, \ t \in \sS,\ \sum_{\si_t(x) = y} \om(t, x)
= 1;$
\item For each $t \in \sS,$ the map $ x \to \om(t, x)$ is
continuous;
\item For each $s,\ t \in \sS$ and $x \in X,\ \om$ satisfies the
cocycle identity
\[ \om(s+t, x) = \om(s, x)\, \om(t, \si_s(x)) .\]
\end{enumerate}
\end{definition}

If the dynamical system $(X, \si, \sS)$ admits a cocycle, then given
$ x \in X$ one can define an orbit cocycle $\mu_x$ by letting
$\mu_x(t, y) = \sqrt{\om(t, \si_t(y))},$ for $t \in \sS$ and $y$ in
the orbit of $x$.
\begin{example} \label{e:localhomeo}
\cite{Peters:semigroup} considers abelian semigroup actions on a
compact metric space by continuous, surjective, locallly injective
maps. Proposition 2 gives necessary and sufficient conditions for a
$\bbZ^k_+$ actions to admit a cocycle, and Example 5 of
\cite{Peters:semigroup} is an action of the non-negative dyadic
rationals on a compact metric space by local homeomorphisms which
admits a cocycle.
\end{example}

\subsection{LEFT REGULAR ORBIT REPRESENTATIONS}
We would like to establish the existence of a class of orbit cocycles which we will call \emph{left regular orbit cocycles}.
To do this, we need to impose a restriction on the dynamical system $(X, \si, \sS).$ If a point $x \in X$ has the property that
for each $y$ in the orbit of $x,\ \{ t \in \sS:\ \si_t(x) = y\}$ is finite, we will say that $x$ has the \emph{finite stability property}.
In case $\sS$ is a group $\sG$, this is equivalent to saying that the stability subroup $\sG_x$ is finite.  However, in our case,
card$\{t \in \sS:\ \si_t(x) = y\}$ could depend on the point $y$, and indeed, these cardinalities need not be bounded.

\begin{example} \label{e:finitestability}
Let $\bbZ^+ = \bbN \cup \{0\}$ and $\sS = \{ (m, n) \in (\bbZ^+)^2: \ 0 \leq m \leq n \}.$  $\sS$ is an abelian semigroup under
coordinatewise addition.  Let $X_1$ be a copy of $-\bbN \cup \{0\},$ and $X_2$ a copy of the integers, both with the discrete topology. These copies are chosen so that $X_1, \ X_2$ are disjoint.
Let $X = X_1 \cup X_2 \cup \{x_{\infty} \}$, with the one point compactification topology.
Then $X$ is metrizable.

Let $\sS$ act on $X$ as follows: $x_{\infty}$ will be fixed under all $\si_{m, n}.$  For $x \in X_1,$ define
\[ \si_{m, n}(x) =
\begin{cases}
x + m, \text{ if } x + m \leq 0 \\
0 \text{ otherwise}
\end{cases} \]
and for $x \in X_2,$ let $\si_{m, n}(x) = x + n.$  Then each map $\si_{m, n}$ is surjective, and continuous in the one point compactification topology.
None is a homeomorphism, except for $\si_{0, 0},$ the identity.

Now if $x \in X_2,$ then $x$ has the finite stability property. Indeed, let $y$ be in the orbit of $x$, so $y = x + k$ for some non-negative integer $k.$
Then $\si_t(x) = y$ iff $t = (m, k)$ for some $m,\ 0 \leq m \leq k.$ This is finite for each such $y$, but there is no upper bound on the cardinalities.
\end{example}

Let us see how the two classes of representations $\pi_{x, \ga}$ and
$\rho_{x, \mu}$ are related by constructing a special cocycle $\mu.$

 Fix a point $x \in X$ which has the finite stability property . Define an
equivalence relation $\sim$ on the semigroup $\sS$ by $s \sim t$ if
$\si_s(x) = \si_t(x),$ and let $[s]$ denote an equivalence class.
Define a map $q: \sS \to \sS(x)$ by $q(s) = \si_s(x).$ Then $q$ is a
one-to-one surjective map of the set of equivalence classes
$\sS/\sim$ to the orbit $\sS(x).$

Let $\sH_0^0$ be the subspace of all (finite) linear combinations
$\eta = \sum a_s \xi_s$ for which $ \sum a_s \xi_{q(s)} = 0.$ Note
that any such sum is the sum of elements $\sum a_s \xi_{q(s)}$ for
which the $s$ appearing in the sum belong to the same equivalence
class, and the sum of the coefficients $\sum a_s = 0.$

For $y $ in the orbit of $x$, let $\sH(y) = \text{span}\{ \xi_t :\ \si_t(x) = y\}$ By the finite stability property of $x$, this subspace is finite dimensional, hence equal to
its closure in $\ell_2(\sS).$  Then $\ell_2(\sS)  =\{ \oplus_{y \in \sS(x)} \sH(y) \}^{-} .$

Let $\sH_0(y) = \sH_0^0 \cap \sH(y).$ This is precisely the codimension one subspace consisting of all  linear combinations $\sum a_t \xi_t$
where $\xi_t \in \sH(y)$ and $\sum a_t = 0.$  Since the subspace is finite dimensional, it is closed in $ \sH(y).$

\begin{lemma} \label{l:invsub} The linear space $\sH_0^0$ is
invariant under the maps $\pi_{x, \ga}(f), \ f \in C(X)$, and under
$\pi_{x, \ga}(S_t), \ t \in \sS.$ Hence the closure of $\sH_0^0,$
which we denote by $\sH_0$, is invariant under $\pi_{x, \ga}(F),$
for $F \in \sA_0,\ \ga \in \Ga.$
\end{lemma}

\begin{proof} Let $f \in C(X)$ and $\eta \in \sH_0^0$. It is enough to show that the subspace
$\sH_0(y),\ y \in \sS(x) $ is mapped to itself under $\pi_{x, \ga}(f).$ So we may assume $\eta \in \sH_0(y),$
 $ \eta = \sum
a_j \xi_{s_j}$ where  $\sum a_j = 0.$ Then
\begin{align*}
\pi_{x, \ga}(f)\eta &= \sum f(\si_{s_j(x)}) a_j \xi_{s_j} \\
    &= \sum f(y) a_j \xi_{s_j} \\
    &\in \sH_0(y)
\end{align*}

Now $\pi_{x, \ga}(S_t)$ does not map the subspace $\sH_0(y)$ to itself, but maps $\sH_0(y) $ to some $\sH_0(y').$
For $t \in \sS,$
\begin{align*}
\pi_{x, \ga}(S_t)\eta &= \sum \langle\ga, t\rangle a_j \xi_{t+s_j} \\
    &\in \sH_0^0
\end{align*}
because if $s_j$ belong to the same equivalence class, then so do
the elements $s_j + t$, since
\[ \si_{s_j + t}(x) = \si_t\circ\si_{s_j}(x) = \si_t(\si_s(x)) \]
where $[s_j] = [s]$ for all $j$.
\end{proof}

\begin{lemma} \label{l:orbitsubspace}
$\sH_0 \cap \sH(y) =  \sH_0(y) .$
\end{lemma}

\begin{proof}
We can write $\sH_0^0$ as the algebraic direct sum of the finite dimensional, orthogonal subspaces $\sH_0(y)$ as $y$ runs through $\sS(x).$ The closure
of $\sH_0^0$,  $\sH_0,$ is thus the $\ell_2$ direct sum of these orthogonal subspaces.
\end{proof}

\begin{remark} If $\sH(y)$ were not finite dimensional, then  $\sH_0(y)$ is dense in $\sH(y).$ In that case, $\sH_0 \supset \sH(y)$ and the
conclusion of the Lemma fails.
\end{remark}

Let $Q$ denote the orthogonal projection of $\ell_2(\sS)$ onto the
subspace $\sH_0.$  Let $\sH_1 = Q^{\perp}(\ell_2(\sS)).$ Observe that for any basis vector $\xi_t \in \ell_2(\sS),\ Q^{\perp}\xi_t \neq 0.$ Indeed, if $Q^{\perp}\xi_t = 0,$
then $\xi_t \in \sH_0$ and in fact $\xi_t \in \sH_0(y)$ where $ y = \si_t(x).$ But if $\sH_0(y)$ contained one basis vector, it would then contain all basis vectors in $\sH(y)$ and
hence $\sH_0(y)$ would coincide with $\sH(y)$,  which is not the case.

Since $\sH_0$ is invariant, we can define the representation
$\pi_{x, \ga}^0$ to be the restriction of $\pi_{x, \ga}$ to the
subspace $\sH_0.$ The subspace $\sH_1$ need not be invariant, but we
can define the representation $\pi_{x, \ga}^1$ by
\[  \pi_{x, \ga}^1(F) = Q^{\perp} \pi_{x, \ga}(F)|\sH_1 .\]
Note that $Q^{\perp}\xi_s \neq 0$ for all $s \in \sS.$

For simplicity of notation, if $\ga = 1$ is the trivial character,
write $\pi_{x, 1} = \pi_x,\ \pi_{x, 1}^1 = \pi_x^1.$

\begin{definition} \label{d:regorbitcocycle}
Fix $x \in X$ which has the finite stability property.
Define an orbit cocycle $\mu$ by setting, for $y \in \sS(x),\ t \in
\sS,$
\[ \mu(t, y) =
\frac{||\pi_x^1(S_t)Q^{\perp}\xi_u||}{||Q^{\perp}\xi_u||} =
\frac{||Q^{\perp}\xi_{t+u}||}{||Q^{\perp}\xi_u||}\] if $y =
\si_u(x).$ This is well-defined, for if $y = \si_{u'}(x),$ then
$Q^{\perp}\xi_u = Q^{\perp}\xi_{u'}.$ We call $\mu$ the \emph{left
regular orbit cocycle.}
\end{definition}

\begin{lemma} \label{l:regorbitcocycle} $\mu$ satisfies the two
conditions of Definition~\ref{d:orbitcocycle}, and hence is an orbit
cocycle. Furthermore, $\mu(s, y) \neq 0$ for all $s \in \sS$ and $y
\in \sS(x).$
\end{lemma}

\begin{proof} Let $s,\ t \in \sS$ and $y \in \sS(x),\ y = \si_u(x).$
Then
\begin{align*}
\mu(t, y)\, \mu(s, \si_t(y)) &=
\frac{||Q^{\perp}\xi_{t+u}||}{||Q^{\perp}\xi_u||}\,
\frac{||Q^{\perp}\xi_{s+t+u}||}{||Q^{\perp}\xi_{t+u}||}
\\
    &= \frac{||Q^{\perp}\xi_{s+t+u}||}{||Q^{\perp}\xi_{u}||} \\
    &= \mu(s+t, y),
\end{align*} verifying the cocycle identity.

Suppose $u_j,\ j = 1, \dots, n$ are elements of $\sS$ such that if
$y_j = \si_{u_j}(x) \in \sS(x)$ are distinct, and $\si_t(y_j) =
\si_t(y), \ 1 \leq j \leq n,$ where $y = y_1 = \si_u(x)$ and $ u =
u_1.$

The vectors $U_j = \frac{1}{||Q^{\perp}\xi_{u_j}||}\,
Q^{\perp}\xi_{u_j}$ are mutually orthogonal unit vectors, and $\xi =
\sum a_j U_j$ is a unit vector if $a_j \in \bbC$ satisfy
$\sum_{j=1}^n |a_j|^2 = 1.$ Now
\begin{align*}
\pi_x^1(S_t) \xi &= \left(\sum
\frac{a_j}{||Q^{\perp}\xi_{u_j}||}\right)Q^{\perp}\xi_{t+u} \\
    &= \left(\sum a_j\,
    \frac{||Q^{\perp}\xi_{t+u}||}{||Q^{\perp}\xi_{u_j}||}\right)\frac{1}{||Q^{\perp}\xi_{t+u}||}Q^{\perp}\xi_{t+u}
\end{align*}
Since $\pi_x^1(S_t)$ is contractive, $||\pi_x^1(S_t)\xi|| \leq 1 .$
Hence, the scalar $|\sum a_j\,
\frac{||Q^{\perp}\xi_{t+u}||}{||Q^{\perp}\xi_{u_j}||}\, | \leq 1,$
for all choices of $a_j$ such that $\sum_{j=1}^n |a_j|^2 = 1.$ By
Cauchy-Schwarz, this implies that
\[ \sum \left(\frac{||Q^{\perp}\xi_{t+u}||}{||Q^{\perp}\xi_{u_j}||}\right)^2
\leq 1 . \]  In other words,
\[ \sum_j \mu(t, y_j)^2 \leq 1 .\]

Finally, $\mu$ is never zero since $Q^{\perp}\xi_s \neq 0$ for all $
s \in \sS.$

\end{proof}

With $\mu$ the left regular orbit cocycle, define $W: \ell_2(\sS(x))
\to \sH_1$ as follows: if $y \in \sS(x),$ say $ y = \si_s(x),$ set
$W\xi_y = \frac{1}{||Q^{\perp}\xi_s||}\, Q^{\perp}\xi_s.$  Then $W$
maps an orthonormal basis of $\ell_2(\sS(x))$ onto an orthonormal
basis of $\sH^1.$ We compute
\begin{align*}
W^* \pi_{x, \ga}^1(S_t)W \xi_y &= W^*\pi_{x,
\ga}^1(S_t)\,\frac{1}{||Q^{\perp}\xi_s||} Q^{\perp}\xi_s \\
    &= \langle\ga, t\rangle\frac{1}{||Q^{\perp}\xi_s||}W^* Q^{\perp}\xi_{t+s} \\
    &= \langle\ga, t\rangle\frac{||Q^{\perp}\xi_{t+s}||}{||Q^{\perp}\xi_s||}
    W^* \frac{1}{||Q^{\perp}\xi_{t+s}||}\, \xi_{t+s} \\
    &= \rho_{x, \ga\mu}(S_t) \xi_{\si_t}(y)
\end{align*}
Also, a straightforward calculation shows that $W^*\pi_{x, \ga}^1(f)
W \xi_y = \rho_{x, \ga\mu}(f)\xi_y .$  This proves

\begin{corollary} \label{c:regorbitrep}
 $W^*\pi_{x, \ga}^1(F)W = \rho_{x, \ga\mu}(F),$
where $F \in \sA_0,\ \ga \in \Ga,$ and $\mu$ is the left regular
orbit cocycle. Thus,
\[ ||\rho_{x, \ga\mu}(F)|| \leq ||\pi_{x, \ga}(F)|| .\]
\end{corollary}

%%%%%%%%%%%%%%%%%%%%%%%%%%%%%%%%%%%%%%%%%%%%%%%%%%%%%%%%%%%%%%%%%%%%%%%%%%%%%%%%%%%%%%%%%%
%%%%%%%%%%%%%%%%%%%%%%%%%%%%%%%%%%%%%%%%%%%%%%%%%%%%%%%%%%%%%%%%%%%%%%%%%%%%%%%%%%%%%%%%%%
%%%%%%%%%%%%%%%%%%%%%%%%%%%%%%%%%%%%%%%%%%%%%%%%%%%%%%%%%%%%%%%%%%%%%%%%%%%%%%%%%%%%%%%%%%%%

\section{EXTENSIONS OF SEMIGROUP DYNAMICAL SYSTEMS}
\label{s:extension}

\begin{definition} \label{d:ext}
Given a dynamical system $(X, \si, \sS)$ we say that the dynamical
system $(Y, \be,\sS)$ is an extension of $(X, \si, \sS)$ if there is
a continuous surjection $p: Y \rightarrow X$ such that the diagram
\[ \begin{CD}
Y    @>\be_s>>    Y \\
@VpVV       @VpVV \\
X   @>\si_s>>     X
\end{CD} \]
commutes for every $s \in \sS$.  We call $p$ the extension map of
$(Y, \be, \sS)$ over $(X, \si, \sS)$.
\end{definition}

We say that an extension $(Y, \be, \sS)$ is a \emph{homeomorphism
extension} of $(X, \si, \sS)$ if the maps $\be_s$ are homeomorphisms
for all $s \in \sS$. We now provide a procedure for producing a
canonical homeomorphism extension of $(X, \si, \sS)$.

Let $\sG = \sS - \sS$ be the group generated by the abelian
semigroup $\sS.$ (Recall that $\sS$ is a semigroup with
cancellation.) Define a partial order on $\sG$ by $h < g$ if $ g - h
\in \sS.$ Let $X_g = X$ for all $g \in \sG.$  If $ g - h = u \in
\sS$ let $\si_u$ map $X_g \to X_h.$  Then the commutativity
conditions for an inverse system are satisfied, so the inverse limit
(or projective limit) of the inverse system exists.  Denote the
inverse limit by $\wti{X}.$

\begin{proposition} \label{p:inversesys}
$\wti{X} = \{ (x_g)_{g \in \sG} \in \Pi X_g :\ x_h = \si_u(x_g)$ for
all $ h < g \in G,$ with $ u = g - h $\}.
\end{proposition}

\begin{proof} This is \cite[Proposition 16-6.4]{Dauns}.
\end{proof}

We now show that there is a homeomorphism $\wti{\si}_t,$ for each $t
\in \sS .$
 Let $\wti{\si}_t$ be the map
$\wti{\si}_t((x_g)_{g\in \sG}) = (\si_t(x_s))_{s\in \sG},$ and let
$p: \wti{X} \to X$ be the map $p((x_s)_{s\in \sG}) = x_0$ (where $0$
is the identity of $\sG$.)

\begin{proposition} \label{p:homext} $(\wti{X}, \wti{\si}, \sS)$ is a
dynamical system for which the $\wti{\si}_t$ are homeomorphisms, for
all $t \in \sS.$ Furthermore, the diagram
\[ \begin{CD}
\wti{X}    @>\wti{\si}_t>>    \wti{X} \\
@VpVV       @VpVV \\
X   @>\si_t>>     X
\end{CD} \]
commutes, so that $(\wti{X}, \wti{\si}, \sS)$ is a homeomorphism
extension of $(X, \si, \sS).$
\end{proposition}

\begin{proof}
We first see that $\wti{\si}_t$ is surjective.  Indeed, let
$((y_g)_{g \in \sG}) \in \wti{X},$ and set $x_g = y_{g+t}.$ Then
$(x_g)_{g\in \sG} \in \wti{X},$ and $\wti{\si}_t((x_g)_{g\in \sG}) =
(y_g)_{g\in \sG}.$

To show injectivity suppose \[ (x_g)_{g\in \sG},\ (x_g')_{g \in \sG}
\in \wti{X}\]  and \[ \wti{\si}_t((x_g)_{g\in \sG}) =
\wti{\si}_t((x_g')_{g\in \sG}). \] Then for all $g \in \sG,\ x_{g-t} =
x_{g-t}'.$  Hence $(x_g)_{g\in \sG} = (x_g'){g\in \sG}.$

\end{proof}

\begin{corollary} \label{c:homext} $\sG$ acts as a group of
homeomorphisms on $\wti{X}.$
\end{corollary}

\begin{proof} Let $g \in \sG.$  Since $\sS - \sS = \sG,\ g$ can be
written as $s - t,$ for $s,\ t \in \sS.$ Define
\[ \wti{\si}_g = \wti{\si}_s \circ \wti{\si}_t^{-1} .\]
We show this is well defined. If also $g = s' - t',$ then $s + t' =
s' + t.$  Hence, $\wti{\si}_{s+t'} = \wti{\si}_{s'+t} .$ From this
we obtain $ \wti{\si}_s \circ \wti{\si}_t^{-1} = \wti{\si}_{s'}\circ
\wti{\si}_{t'}^{-1}.$

\end{proof}

Our next goal is to show that the extension $(\wti{X}, \wti{\si},
\sS)$ is a minimal extension of $(X, \si, \sS)$ in a sense we will
make precise.

\begin{lemma} If $\si_t$ is a homeomorphism for all $t \in \sS$, then the
map $p: \wti{X} \rightarrow X$ is a homeomorphism. Hence the
dynamical systems $(X, \si, \sS)$ and $(\wti{X}, \wti{\si}, \sS)$
are conjugate.
\end{lemma}

\begin{proof} We need only show that the map $p$ is injective, since
by Definition~\ref{d:ext} it is a continuous surjection.  So assume
that $p((x_s)_{s \in \sS}) = p((y_s)_{s \in \sS})$.  In particular
$x_0 = y_0$.  Now since $\si_t$ is a homeomorphism we have that $x_s
= \si_s^{-1}(x_0) = \si_s^{-1}(y_0) = y_s$ and hence $p$ is
injective. That the systems are conjugate follows from the
commutative diagram for the notion of extension.
\end{proof}

\begin{definition} \label{d:minext} Consider an extension $(Y, \be,
\sS)$ of $(X,\si, \sS)$ via an extension map $r$.  We say that an
extension $(Z, \vpi, \sS)$ of $(X, \si,\sS)$ lies between $(Y, \be,
\sS)$ and $(X, \si_s,\sS)$ if the following diagram
\[ \begin{CD}
Y    @>\be_t>>    Y \\
@VpVV       @VpVV \\
Z   @>\vpi_t>>     Z \\
@VqVV       @VqVV \\
X   @>\si_t>>     X
\end{CD} \]
commutes for all $t \in \sS,$ and $q \circ p = r,$ where $p$ and $q$
are the extension maps as in the diagram.

We say the extension  $(Y,\be, \sS)$ of $(X,\si, \sS)$ via an
extension map $r$ is a \emph{homeomorphism extension} if the maps
$\be_t,\ t \in \sS$ are homeomorphisms, for $ t \in \sS.$ Finally,
we call a homeomorphism extension $(Y,\be, \sS)$ of $(X,\si, \sS)$
\emph{minimal} if for any dynamical system $(Z, \vpi, \sS)$ that
lies between the two systems as in the diagram, the extension map
$p$ is a homeomorphism, and hence $(Y, \be, \sS)$ and $(Z, \vpi,
\sS)$ are conjugate systems.
\end{definition}

We refer to the homeomorphism extension $(\wti{X}, \si, \sS)$ of
$(X, \si, \sS)$ as the \emph{canonical} homeomorphism extension.

\begin{lemma} \label{l:minext}
The canonical homeomorphism extension of $(X, \si,\sS)$ is a minimal
extension.
\end{lemma}

\begin{proof} Assume that we have the following commutative diagram
\[ \begin{CD}
\wti{X}    @>\wti{\si_t}>>    \wti{X} \\
@VpVV       @VpVV \\
Z   @>\vpi_t>>     Z \\
@VqVV       @VqVV \\
X   @>\si_t>>     X
\end{CD} \]
where $\wti{\si_t}$ and $\vpi_t$ are homeomorphisms for every $t$
and $p$ and $q$ are surjections with $ q \circ p ((x_s)_{s \in \sS})
= x_0$.  By the preceding lemma we know that $\wti{Z}$ and $Z$ are
conjugate and hence we will show that $\wti{Z}$ and $\wti{X}$ are
conjugate.

For notational purposes we will interchange the notations $x =
(x_s)_{s\in \sS}$ for an element of $\wti{X}$ as necessary. We
define a map $\Ga: \wti{X} \rightarrow \wti{Z}$ by $\Ga ((x_s)_{s\in
\sS}) = (\vpi_s^{-1}(p(x)))_{s \in \sS}$. The map $\Ga$ is clearly
continuous.  On the other hand notice that $\vpi_t(
\vpi_{s+t}^{-1}(p(x)) ) = \vpi_s^{-1}(p(x))$ and hence $\Ga(x) \in
\wti{Z}$.

Now if $(y_s)_{s\in \sS} \in \wti{Z}$ then there exists $x \in
\wti{X}$ such that $p(x) = y_0$ since $p$ is surjective.  Also $y_s
= \vpi_s^{-1}(y_0)$ and hence $\Ga(x) = (y_s)_{s \in \sS}$ and hence
$\Ga$ is onto.

To see that $\Ga$ is one-to-one consider the map $\Pi: \wti{Z}
\rightarrow \wti{X}$ given by \[ \Pi((y_s)_{s \in \sS}) = (q(y_s))_{s
\in \sS}. \]  Notice that $\vpi_t(q(y_{s+t})) = q(\vpi_t(y_{s+t})) =
q(y_s)$ and so the map $\Pi$ does map into $\wti{X}$.  Now we see
that
\begin{align*} \Pi \circ \Ga (x) & = (q(\vpi_s^{-1}(p(x))))_{s
\in \sS} \\ & = (q(p \circ \wti{\si_s}^{-1}(x)))_{s \in \sS}\\ &= (q
\circ p((x_{t+s})_{t \in \sS}))_{s \in \sS} \\ &= (x_s)_{s\in \sS} =
x.
\end{align*}  It follows that $\Ga$ is one-to-one and hence
$\Ga$ is a homeomorphism.  The conjugacy follows immediately from
the commutative diagram. \end{proof}

The next theorem now follows immediately.

\begin{theorem} \label{t:homeoext} The dynamical system $(X, \si,
\sS)$ has a minimal homeomorphism extension, which is unique up to
conjugacy.
\end{theorem}

\subsection{DUALIZING THE CANONICAL HOMEOMORPHISM EXTENSION}
\label{s:injlim}

Let $\sA = C(X)$, and $\al_s$ be the endomorphism $\al_s(f) = f\circ
\si_s,\ s \in \sS,\ f \in C(X).$  Define a partial order on $\sS$ by
$ t \succ s$ if there exists $u \in \sS$ such that $t = u+s.$ Now
the diagram
\[
\begin{CD}
\sA   @>\al_s>>   \sA \\
@V\al_{t}VV        @V\al_uVV \\
\sA   @=          \sA
\end{CD}
\]
commutes, so we can form the inductive system $(\sA, \al_s)$ with
respect to the order $\succ.$ Let $\wti{\sA} = \vil (\sA_s, \al_s)$
where $\sA_s = \sA$ for all $s$, and let $\io_s$ be the canonical
embeddings of $\sA \to \wti{\sA}.$ Thus we have the commutative
diagram
\[
\begin{CD}
\sA    @>\al_s>>    \sA \\
@V\io_{t+s}VV       @V\io_tVV \\
\wti{\sA}   @=     \wti{\sA}
\end{CD}
\]

Now we define $\wti{\al_s}: \wti{\sA} \to \wti{\sA}$ as follows: if
$\wti{a} \in \wti{\sA}$ there exists a $t \in \sS$ and $ a \in \sA$
such that $\wti{a} = \io_t(a).$ Define
\[ \wti{\al_s}(\wti{a}) = \io_t(\al_s(a)) .\]
This is well defined by the commutativity of the diagrams. Since
$\al_s,\ s \in \sS,$ is linear and injective, the same is true for
$\wti{\al_s}.$ We show $\wti{\al_s}$ is invertible.

Let $\wti{a} \in \wti{\sA}$ be given; say $\wti{a} = \io_t(a)$ for
some $t \in \sS,\ a \in \sA.$ We can assume $ t \succ s,$ say $ t =
u+s$ for some $u \in \sS.$ Then
\[ \wti{a} = \io_t(a) = \io_u\circ \al_s(b)\]
for some $b \in \sA.$  Thus, $\wti{a} = \wti{\al}_s(\wti{b})$ where
$\wti{b} = \io_u(b).$

Now the mappings $\al_s,\ \io_t,\ ( s, \ t \in \sS)$ are isometric
and $^*$-maps (i.e. $\al_s(\ol{a}) = \ol{\al_s(a)}$ ) hence
$\wti{\sA}$ is the direct limit of $C^*$-algebras, so that the
completion of $\wti{\sA}$ is a commutative C$^*$-algebra, $C(Z)$.
The automorphisms $\wti{\al}_s$ are isometric on $\wti{\sA}$, hence
extend to automorphisms, also denoted $\wti{\al}_s$, of $C(Z).$
Thus, by the Banach-Stone Theorem, there is a homeomorphism $\vpi_s$
of $Z$ such that $\wti{\al}_s(f) = f\circ\vpi_s,\ f \in C(Z),\ s \in
\sS$.

Let $j$ be the embedding $C(X) \mapsto C(\wti{X})$ given by $j(f) =
f\circ p$ where $p: \wti{X} \to X$ is the canonical map. Now for $s
\in \sS$ let $\be_s: \sA \to C(\wti{X})$ be the map $\be_s(f) =
j(f)\circ\wti{\si}_{-s}.$ Then the diagram
\[
\begin{CD}
\sA    @>\al_u>>    \sA  \\
@V\be_sVV          @V\be_{s+u}VV \\
C(\wti{X})   @=      C(\wti{X})
\end{CD}
\] commutes.
Thus, by properties of direct limits, there is a star homomorphism
$\Psi: \wti{\sA} \to C(\wti{X}).$ Since the maps $\be_s$ are
isometric, so is $\Psi,$ hence $\Psi$ extends to a map (also denoted
$\Psi$) of $C(Z) \to C(\wti{X}).$

Now the embedding $C(Z) \to C(\wti{X})$ yields a map $p: \wti{X} \to
Z$ as follows: let $\wti{x}$ be a pure state on $C(\wti{X}),$ which
we identify with a point of $\wti{X}.$ Restricting $\wti{x}|_{C(Z)}$
yields a pure state of $C(Z),$ which is canonically identified with
a point of $Z$.

We observe that the diagram
\[
\begin{CD}
\wti{X}      @>\wti{\si}_s>>      \wti{X} \\
@VpVV                           @VpVV  \\
Z           @>\vpi_s>>          Z
\end{CD}
\] commutes. By the minimal extension property of $\wti{X}$ (cf Lemma~\ref{l:minext}), $p$ is a
homeomorphism. Thus, $C(\wti{X})$ is the (completion of) the direct
limit of the directed system $(C(X),\al_s)$.

\section{THE C$^*$-ENVELOPE}

\begin{theorem} \label{t:cstarenv} The C$^*$-envelope of the left regular algebra
$\sA(X, \sS)$ is the crossed product $C(\wti{X})\rtimes_{\ti{\al}}
\sG.$
\end{theorem}

\begin{proof} Define representations $\wti{\pi}_{\hat{x}, \ga}$ for $x
\in X$ and $\ga \in \Ga$ of the crossed product
$C(\wti{X})\rtimes_{\ti{\al}}\sG $ as follows. Let $\hat{x}$ be the
subset $p^{-1}(x) \subset \wti{X},$ where $p$ is the map $\wti{X}
\to X$ given in Definition~\ref{p:homext}. The Hilbert space is
$\ell_2(\sG(\hat{x})),$  where $\sG(\hat{x})$ denotes the union of
the orbits $\sG(\ti{x})$ for $\ti{x} \in \hat{x}.$ If $U_g$ is the
unitary element in the crossed product associated with the
homeomorphism $\wti{\si}_g,$ the representation is given by
\[\wti{\pi}_{\hat{x}, \ga}(U_g)\xi_{\ti{x}} =  \langle\ga,
g\rangle\xi_{\wti{\si}(\ti{x})}\] where $\xi_{\ti{x}}$ is the
function in $\ell_2(\sG(\hat{x}))$ which is $1$ at $\ti{x}$ and zero
elsewhere. And for $\ti{f} \in C(\wti{X}),\ \wti{\pi}_{\hat{x},
\ga}(\ti{f})\xi_{\ti{x}}  = \ti{f}(\ti{x})\xi_{\ti{x}}.$

Since the direct sum of the representations $\ti{f} \to
\ti{f}\large| \hat{x}$ (that is, the restriction of $\ti{f}$ to the
subset $\hat{x} \subset \wti{X}$) of $C(\wti{X})$ is faithful, it
follows that the the supremum of the norms of the representations
$\wti{\pi}_{\hat{x}, \ga}$ is faithful on the crossed product, since
$\sG$ is abelian, hence amenable. (cf \cite[7.7.5]{Pet})  Indeed,
this holds even if $\ga$ is taken to be the trivial character.

Since $C(X)$ is embedded in $C(\wti{X})$ by the map $j(f) = f\circ
p,$ it follows that $\wti{\pi}_{\hat{x}, \ga}(j(f))\xi_{\ti{y}} =
f(y)\xi_{\ti{y}} $ for $\ti{y} \in \sS(\ti{x})$ with $p(\ti{y}) = y
\in X,$ since $j(f)$ is constant on the subset $\hat{y} \subset
\wti{X},$ and that constant is $f(y).$

Let $F \in \sA_0,$ say $F = \sum S_s f_s$ (where the sum if finite),
let $\wti{F} = \sum U_s j(f_s).$ Then we have that
\[ ||\wti{\pi}_{\hat{x}, \ga}(\wti{F})|| = ||\pi_{x, \ga}(F)|| .\]
It follows that equality holds for $F \in \sA(X, \sS)$ and hence
that the embedding of the left regular algebra $\sA(X, \sS)$ into
the crossed product is completely isometric. We will also denote
this embedding by $j,$ which is consistent if we view $C(X)$ as a
subalgebra of $\sA(X, \sS)$ and $C(\wti{X})$ as a subalgebra of the
crossed product.

To complete the proof, suppose $\sB$ is the C$^*$-envelope of
$\sA(X, \sS),$ and let $k: \sA(X, \sS) \to \sB$ be the completely
isometric embedding.  Then there is surjective C$^*$-homomorphism
$\Phi: C(\wti{X})\rtimes_{\ti{\al}} \sG \to \sB$ such that the
diagram
\[
\begin{CD}
\sA(X, \sS) @>j>>   C(\wti{X})\rtimes_{\ti{\al}} \sG \\
@V{id}VV  @V{\Phi}VV \\
\sA(X, \sS) @>k>> \sB
\end{CD}
\] commutes. It remains to show that $\Phi$ is an isomorphism. Suppose,
to the contrary, there is an element $H \in \text{ker}(\Phi).$ We
may suppose $H$ has norm $1$. $H$ can be approximated by an element
$G$ with $||G - H|| < \frac{1}{4},$ where $G = \sum_{i=1}^n U_{g_i}
h_i$ with $ h_i \in C(\wti{X})$.

From Section~\ref{s:injlim} there exist $f_i \in C(X)$ such that
$||j(f_i) - h_i|| < \frac{1}{4n},\ 1 \leq i \leq n.$  Thus if $F =
\sum U_{g_i} j(f_i),$ then $||F - G|| < \frac{1}{4}.$

Now express $g_i = s_i - t_i,$ where $s_i,\ t_i \in \sS,\ 1 \leq i
\leq n.$ Let $ U = U_{t_1} \cdots U_{t_n} = U_{t_1 +\cdots + t_n}.$
Then $FU \in j(\sA(X, \sS)),$ and since $U$ is unitary in the
crossed product, $||FU|| = ||F||.$

Now $||H - F|| < \frac{1}{2}, $ so that $||\Psi(H - F)|| <
\frac{1}{2}.$  Since $||H|| = 1,$ this implies $||F|| = ||FU||
> \frac{1}{2}.$ Hence,
\begin{align*}
||\Psi(HU - FU)|| &\leq ||\Psi(H - F)|| ||\Psi(U)|| \\
    &\leq ||\Psi(H - F)|| \\
    &< \frac{1}{2}
\end{align*}
whereas, since $\Psi(H) = 0,$
\begin{align*}
||\Psi(HU - FU)|| &= ||\Psi(FU)|| \\
    &> \frac{1}{2}
\end{align*}
since, by the defining property of the C$^*$-envelope, $\Psi$ is
completely isometric on $j(\sA(X, \sS)).$ This contradiction shows
that the kernel of $\Psi$ is trivial, and so the crossed product is
the C$^*$-envelope.

\end{proof}

\begin{corollary} \label{c:dilationn} Given $x \in X,\ \ga \in \Ga,$
\begin{enumerate}
\item
The semigroup $\pi_{x, \ga}(S_s) \ (s \in \sS)$ of commuting
isometries dilates to a commuting semigroup of unitaries;
\item
Assume $x$ has the finite stability property. Then semigroup $\rho_{x, \ga}(S_s)\ (s \in \sS)$ of commuting
contractions dilates to a commuting semigroup of unitaries.
\end{enumerate}
\end{corollary}

\begin{theorem} \label{t:complcontr}  There is a completely contractive
representation \newline
\mbox{$\Pi: C(X) \rtimes_{\si} \sS
\rightarrow \sA(X, \sS).$}
\end{theorem}

\begin{proof} Let $ F \in \sA_0.$ The norm of $F$ as an element of
the semicrossed product $C(X) \rtimes_{\si} \sS$ is given as the
supremum over all representations $||\pi(F)||$ which satisfy the
three properties of Definition~\ref{d:rep}. The norm of $F$ as an
element of the left regular algebra $\sA(X, \sS)$ is given as the
supremum over a subset of these representations. Since the
semcrossed product is the completion of $\sA_0$ in the larger norm,
for $F \in \sA(X, \sS),$ we may take
\[ \Pi(F) = \oplus_{(x, \ga) \in X\times \Ga} \pi_{x, \ga}(F) .\]
This yields a contractive map of the semicrossed product into
$\sA(X, \sS).$

To see the map is completely contractive, the proof of
Theorem~\ref{t:cstarenv} shows that the representation $\pi_{x,
\ga}(F)$ is unitarily equivalent to the restriction of
$\wti{\pi}_{\hat{x}, \ga}(F)$ to an invariant subspace. Since this
is a C$^*$-representation, it is completely contractive, and the
same is true of the direct sum of such representations.  Thus the
map $\Pi$ is completely contractive.
\end{proof}

\begin{remark} \label{r:complcontr}
If the map $\Pi$ is not completely isometric, then it would be
interesting to have examples of representations $\pi$ for which the
norm $||\pi(F)||$ is not dominated by the norm of $F$ in $\sA(X,
\sS).$ Conceivably such representations could be orbit
representations for which the associated orbit cocycle is not the
left regular orbit cocycle.  Of course, the existence of such
cocycles will depend on the semigroup $\sS.$  For, say if $\sS =
\bbN,$ then the semicrossed product norm and the left regular norm
coincide.  More generally, what condition on the dynamical system
$\sS$ is needed to insure that the two norms are different?
\end{remark}

Recall the maps $P_s: \sA(X, \sS) \to \sA(X, \sS)$ defined in Section~\ref{s:semicrossed}
Now, with abuse of notation, we define the conditional expectation $P_0$ on $C(\wti{X})\rtimes_{\ti{\al}}
\sG.$ in the same way it was defined on $\sA_0,$ by
\[ P_0(F) = \int_{\Ga} \tau_{\ga}(F)\, d\ga \]
where now $\tau_{\ga}$ acts on the group $\sG.$

 Note that $P_0$ maps onto the subalgebra $C(\wti{X})U_0.$
If we regard $\sA(X, \sS)$ as a subalgebra of its C$^*$-envelope, then the map $P_0$ of Section~\ref{s:semicrossed} coincides
with the restriction of this map $P_0$ to $\sA(X, \sS).$

\begin{proposition} \label{p:condexp}
$P_0$ is a faithful, completely contractive conditional expectation
of $\sA(\sS, X) \to C(X).$
\end{proposition}

\begin{proof} For $F \in \sA_0,\ F = \sum  S_s \, f_s$ (finite sum),
$P_0(F)  = f_0$ where $S_0 = I$.  So it is evident that
\[ P_0(fF) = fP_0(F) = P_0(Ff) \]
for any $f \in C(X).$

We see that for $F \in \sA_0$ as above,
\[ P_0(F^*F) = \sum |f_s|^2 \] and in particular,
\[ (\ddag)\quad P_0(F^*F) \geq |f_s|^2 = P_s(F)^*P_s(F) \] for any $s$.
So, by continuity of $P_0$ and density of $\sA_0,$ it follows that
$\ddag$ holds for $F \in \sA(\sS, X).$

Now  suppose $F \in \sA(\sS, X)$ and $P_0(F^*F) = 0.$  Then it
follows that $P_s(F) = 0 $ for all $s \in \sS.$  So, by
Proposition~\ref{p:nonzeroproj}, $F = 0.$

The map $F \in \sA(X, \sS) \to \tau_{\ga}(F)$ is completely
isometric.  As $P_0$ is the average of completely isometric maps, it
is completely contractive.
\end{proof}

%\begin{corollary} \label{c:condexp} There is a
%faithful conditional expectation of the semicrossed product $C(X)
%\rtimes_{\si} \sS$ onto $C(X)$. \end{corollary}

%\begin{proof} By the definition of the norm on the semicrossed
%product, there is a contractive map of $C(X) \rtimes_{\si} \sS$ onto
%$\sA(\sS, X).$ The composition of this map with the conditional
%expectation on $\sA(\sS, X)$ is the desired expectation.
%
%\end{proof}

\begin{corollary} There is a completely contractive conditional expectation
\[ C(X)\rtimes_{\si}\sS \to C(X) \]
\end{corollary}

\begin{proof} By Theorem~\ref{t:complcontr} the map $\Pi$ of
$C(X)\rtimes_{\si}\sS $ onto $\sA(X, \sS)$ is completely
contractive, and by Proposition~\ref{p:condexp} the conditional
expectation of $\sA(X, \sS)$ onto $C(X)$ is completely contractive.
The conditional expectation on the semicrossed product is the
composition of the two maps.
\end{proof}

%%%%%%%%%%%%%%%%%%%%%%%%%%%%%%%%%%%%%%%%%%%%%%%%%%%%%%%%%%%%%%%%%%%%%%%%%%%%%%%%%%%%%%%%%%%%%%%%%%%%%%%
%%%%%%%%%%%%%%%%%%%%%%%%%%%%%%%%%%%%%%%%%%%%%%%%%%%%%%%%%%%%%%%%%%%%%%%%%%%%%%%%%%%%%%%%%%%%%%%%%%%%%%%%
%%%%%%%%%%%%%%%%%%%%%%%%%%%%%%%%%%%%%%%%%%%%%%%%%%%%%%%%%%%%%%%%%%%%%%%%%%%%%%%%%%%%%%%%%%%%%%%%%%%%%%%%%

\subsection{SHILOV MODULES} \label{s:shilov}

\begin{definition} \label{d:shilov} Let $\sA$ be an operator
algebra, $\pi: \sA \to \sB(\sH)$ a representation. Then $\pi$ is
said to be a \emph{Shilov} representation if there is a
representation $\Pi$ of the C$^*$-envelope C$^*(\sA)$ in a Hilbert
space $\sK$ containing $\sH$ as a subspace, so that (viewing $\sA$
as a subalgebra of  C$^*(\sA)$, $ \pi(F)$ is the restriction of
$\Pi(F)$ to $\sH, \text{ for all } F \in \sA .$ \cite{Muhly-Solel1}
expresses this in the language of modules: $\sH$ is isomorphic to a
submodule of $\sK$ viewed as an $\sA$-module.

A Hilbert module $\sH$ is said to have a \emph{Shilov resolution} if
there is a short exact sequence of $\sA$ modules
\[ 0 \rightarrow \sK_0 \rightarrow \sK \overset{\Phi}{\rightarrow}
\sH \rightarrow 0 \] where $\sK_0$ and $ \sK$ are Shilov modules.
\end{definition}

Let $\sH_0,\ \pi_{x, \ga}^1,\ \sH_1,\ \pi_{x, \ga}^1$ be the Hilbert
spaces and representations introduced prior to
Definition~\ref{d:regorbitcocycle}. While these were initially
defined as representations of $\sA_0,$ they are uniquely extendible
to representations of $\sA = \sA(X, \sS),$ and it is this context we
consider them here. Following \cite{Muhly-Solel1}, we employ the
language of  Hilbert modules.

For the remainder of this section let us fix $x \in X$ which has the finite stability property, and $\ga \in \Ga.$ View $\sH_0$ as an $\sA$
module via the representation $\pi_{x, \ga}^0,$ $\sH_1$ as an $\sA$
module via the representation $\pi_{x, \ga}^1,$ and $\ell_2(\sS) $
as an $\sA$ module via the representation $ \pi_{x, \ga}.$

\begin{theorem} \label{t:shilov}
\begin{enumerate}
\item $\ell_2(\sS)$ is a Shilov module;
\item $\sH_1$ has a Shilov resolution
\[ 0 \rightarrow \sH_0 \rightarrow \ell_2(\sS)
\overset{Q^{\perp}}{\rightarrow} \sH_1 \rightarrow 0 .\]
\end{enumerate}
\end{theorem}

\begin{proof} (1) Theorem~\ref{t:cstarenv} shows that for $ F \in
\sA,\ \pi_{x, \ga}(F)$ is unitarily equivalent to the restriction of
the representation $\wti{\pi}_{\hat{x}, \ga}(F)$ of the
C$^*$-envelope to an invariant subspace.

(2) Since $\pi_{x, \ga}$ is a Shilov representation of $\sA,$ so is
its restriction to an invariant subspace.  Thus $\sH_0$ is a Shilov
module. Since $\sH_1$ is the quotient space $\ell_2(\sS)/\sH_0,$ it
has a Shilov resolution as given in (2).
\end{proof}

Again fixing $ x \in X$ and $\ga \in \Ga,$ and let $\mu = \mu_x$ be
the left regular orbit cocycle, and $\rho_{x, \ga\mu}$ the
associated representation of the orbit space $\ell_2(\sS(x)),$ which
we view as an $\sA(X, \sS)$ module via this representation.

\begin{corollary} \label{c:shilov}
As a left $\sA$-module, the orbit Hilbert space $\ell_2(\sS(x))$ has
a Shilov resolution.
\end{corollary}

\begin{proof} This follows from Corollary~\ref{c:regorbitrep} in
which it is shown that $\rho_{x, \ga\mu}$ is unitarily equivalent to
$\pi_{x, \ga}^1.$
\end{proof}

\begin{remark} \label{r:dilation} For $x \in X $ as above, the family of commuting contractions $\{ \rho_{x, \ga}(S_t):\ t \in \sS\}$ dilates to a commuting family
$\{\pi_{x, \ga}(S_t):\ t \in \sS\}$ of isometries, which in turn has an extension to a commuting family of unitaries acting on $\ell_2(\sG).$
\end{remark}

\noindent
Acknowledgements.  We would like to thank Ken Davidson for a suggestion in section 4.2.  The second author acknowledges partial support from the National Science Foundation, DMS-0750986

\pagebreak

\bibliographystyle{plain}

\begin{thebibliography}{00}

\bibitem{Arveson-Josephson}
\textsc{W.\ Arveson and K.\ Josephson}, Operator algebras and measure
preserving automorphisms.\ II, \textit{ J.\ Funct.\ Anal.\ }{\bf 4}
(1969) 100--134.

\bibitem{Dauns}
\textsc{J.\ Dauns}, \textit{ Modules and Rings} Cambridge University Press,
Cambridge, 2004.

\bibitem{DFK}
\textsc{K. Davidson, A. Fuller and E. Kakariadis}, Semicrossed Products of Operator Algebras by Semigroups,
{\emph{arXiv}}1404.1904, 2014.

\bibitem{DK:isomorphisms}
\textsc{K.\ Davidson and E.\ Katsoulis}, Isomorphisms between topological
conjugacy algebras, \textit{ J.\ reine angew.\ Math.\ }{\bf 621},
(2008), 29--51.

\bibitem{DK:multivariate}
\textsc{K. \ Davidson and E. \ Katsoulis}, Operator algebras for multivariable dynamics.
\textit{ Mem.\ Amer. \ Math.\ Soc.} 209 (2011), no. 982, viii+53 pp.

\bibitem{DKM}
\textsc{A.\ Donsig, A.\ Katavolos, and A.\ Manoussos}, The Jacobson radical
for analytic crossed products, \textit{ J.\ Funct.\ Anal.} {\bf 187} no.
6, (2001), 129--145.

\bibitem{Exel:endo}
\textsc{R.\ Exel}, A new look at the crossed product of a $C^\ast$-algebra by
a semigroup of endomorphisms, \textit{ Ergodic Theory Dynam.\ Systems }
{\bf 28} (2008), no. 3, 749--789.

\bibitem{Exel-Renault:groupoid}
\textsc{R.\ Exel and J.\ Renault},  Semigroups of local homeomorphisms and
interaction groups, \textit{ Ergodic Theory Dynam.\ Systems} {\bf 27}
(2007), no. 6, 1737--1771.

\bibitem{F}
\textsc{A. Fuller}, Nonself-adjoint semicrossed products by abelian semigroups.
\textit{ Canad.\ J.\ Math.} {\bf 65} (2013), no. 4, 768--782.

\bibitem{KK}
\textsc{E.\ Kakariadis and E.\ Katsoulis}, Semicrossed Products of Operator Algebras and their C$^*$-envelopes,
\textit{ J.\ Funct.\ Anal.} {\bf 262:7} (2012), 3108 -- 3124.

\bibitem{Katsura}
\textsc{T.\ Katsura}, {\em On C$^*$-algebras associated with
C$^*$-correspondences} \textit{ J.\ Funct.\ Anal.} 217 (2004), 366--401.

\bibitem{Muhly-Solel1}
\textsc{P.\ Muhly and B. Solel}, Hilbert Modules over Operator Alebras, \textit{
Memoirs of the A.M.S.} {\bf 117 No. 559} (1995).

\bibitem{Muhly-Solel2}
\textsc{P.\ Muhly and B.\ Solel}, Tensor algebra over $C^*$-correspondences:
representations, dilations, and C$^*$-envelopes, \textit{ J.\ Funct.\
Anal.\ }{\bf 158} (1998) 389--457.

\bibitem{Parrott}
\textsc{S. Parrott},  Unitary Dilations for commuting contractions,
\textit{ Pacific J.\ Math.} 34 (1970), 481--490.

\bibitem{Paulsen}
\textsc{V.\ Paulsen}, \textit{ Completely bounded maps and operator algebras}
Cambridge University Press, Cambridge, 2002.

\bibitem{Peters:semicrossed}
\textsc{J.\ Peters}, Semicrossed products of $C^*$-algebras, \textit{ J.\ Funct.\
Anal.\ }{\bf 59} (1984) 498--534.

\bibitem{Peters:envelope}
\textsc{J.\ Peters}, The $C^*$-envelope of a semicrossed product and nest
representations, {\em Operator structures and dynamical systems,}
197--215, \textit{ Contemp.\ Math.}, 503, 2009.

\bibitem{Peters:semigroup}
\textsc{J.\ Peters},  Semigroups of locally injective maps and transfer operators.
\textit{ Semigroup Forum} 81 (2010), no.\ 2, 255–268.

\bibitem{Pet}
\textsc{Gert Pederson}, \textit{ C$^*$-algebras and their
Automorphism Groups}  Academic Press, 1979.

\end{thebibliography}

\end{document}